\documentclass[a4paper,11pt]{amsart}

\usepackage{amsfonts,amssymb,mathtools}
\usepackage{textcomp, color}

\theoremstyle{plain}

\newtheorem*{thmA}{Theorem A}
\newtheorem*{thmB}{Theorem B}

\newtheorem{thm}{Theorem}[section]
\newtheorem{lem}[thm]{Lemma}
\newtheorem{pro}[thm]{Proposition}

\theoremstyle{definition}

\newtheorem{dfn}[thm]{Definition}

\newtheorem{exa}[thm]{Example}

\begin{document}

\title{On ramification structures for finite nilpotent groups}

\author[\c{S}.\ G\"ul]{\c{S}\"ukran G\"ul}
\address{Department of Mathematics\\ University of the Basque Country UPV/EHU\\
 48080 Bilbao, Spain}
\email{sukran.gul@ehu.eus}

\keywords{Ramification structures, Beauville structures, finite nilpotent groups, finite $p$-groups. \vspace{3pt}}

\thanks{The author is supported by the Spanish Government, grant MTM2014-53810-C2-2-P, and by the Basque Government, grant IT974-16}.

\begin{abstract}
We extend the characterization of abelian groups with ramification structures given by Garion and Penegini in \cite{GP2} to finite nilpotent groups whose Sylow $p$-subgroups have a `nice power structure', including regular $p$-groups, powerful $p$-groups and (generalized) $p$-central $p$-groups. We also correct two errors in \cite{GP2} regarding abelian $2$-groups with ramification structures and the relation between the sizes of ramification structures for an abelian group and those for its Sylow $2$-subgroup.
\end{abstract}

\maketitle

\section{Introduction}

An algebraic surface $S$ is said to be \emph{isogenous to a higher product of curves} if it is isomorphic to $(C_1\times C_2)/G$, where $C_1$ and $C_2$ are curves of genus at least $2$, and $G$ is a finite group acting freely on $C_1\times C_2$. Particular interesting examples of such surfaces are \emph{Beauville surfaces}. These are algebraic surfaces isogenous to a higher product which are rigid.

Groups of surfaces isogenous to a higher product can be characterized by a purely group-theoretical condition: the existence of a `ramification structure'.

%A surface isogenous to a higher product is either of \emph{mixed} or \emph{unmixed} type according as $G$ acts by exchanging the two factors  or diagonally on their product, respectively. In this paper we restrict our considerations to the unmixed case.

\begin{dfn}
Let $G$ be a finite group and let $T=(g_1, g_2, \dots, g_r)$ be a tuple of non-trivial elements of $G$.
\begin{enumerate}
\item  
$T$ is called a \emph{spherical system of generators} of $G$ if $\langle g_1, g_2, \dots, g_r\rangle=G$ and $g_1g_2\dots g_r=1$.
\item 
$\Sigma(T)$ is the union of all conjugates of the cyclic subgroups generated by the elements of $T$:
\[
\Sigma(T)
=
\bigcup_{g\in G} \bigcup_{i=1}^{r} \,
\langle g_i \rangle^g.
\]
Two tuples $T_1$ and $T_2$ are called \emph{disjoint} if 
$
\Sigma(T_1) \cap \Sigma(T_2)=1.
$
\end{enumerate}
\end{dfn}

\begin{dfn}
\label{dfn}
An \emph{(unmixed) ramification structure} of size $(r_1,r_2)$ for a finite group $G$ is a pair $(T_1, T_2)$ of disjoint spherical systems of generators of $G$, where $|T_1|=r_1$ and $|T_2|=r_2$. We denote by $S(G)$ the set of all sizes $(r_1, r_2)$ of ramification structures of $G$.
\end{dfn}

Observe that if $d$ is the minimum number of generators of $G$, spherical systems of generators of $G$ are of size at least $d+1$. Since clearly cyclic groups do not admit ramification structures, it follows that $r_1, r_2 \geq 3$ in Definition \ref{dfn}.

If $r_1=r_2=3$, then ramification structures coincide with \emph{Beauville structures}, which have been intensely studied in recent times; see surveys \cite{bos, fai, jon}.
Knowledge about ramification structures that are not Beauville is very scarce. In 2013, Garion and Penegini \cite{GP} proved that almost all alternating groups and symmetric groups admit such structures.
Soon afterwards, they characterized the abelian groups with ramification structures \cite[Theorem 3.18]{GP2}. 

After abelian groups, the most natural class of finite groups to consider are nilpotent groups. 
As we will see in Proposition \ref{nilpotent general}, a finite nilpotent group admits a ramification structure if and only if so do its Sylow $p$-subgroups. 
The goal of this paper is to extend the characterization of abelian groups with ramification structures to finite nilpotent groups whose Sylow $p$-subgroups have a good behavior with respect to powers. To this purpose, we first study the existence of ramifications structures for finite $p$-groups with a `nice power structure'. In particular,  we generalize Theorem A in \cite{FG}, which determines the conditions for such $p$-groups to be Beauville groups.

 If $G$ is a finite $p$-group of exponent $p^e$,  we call $G$ \textit{semi-$p^{e-1}$-abelian} if for every $x, y\in G$, we have
\begin{equation*}
\label{semi-abelian}
x^{p^{e-1}}=y^{p^{e-1}} \ \
\text{if and only if} \ \
(xy^{-1})^{p^{e-1}}=1.
\end{equation*}

\begin{thmA}
Let $G$ be a finite $p$-group of exponent $p^e$, and let $d=d(G)$. Suppose that $G$ is semi-$p^{e-1}$-abelian.
Then $G$ admits a ramification structure if and only if
$
|\{g^{p^{e-1}} \mid g \in G \}| \geq p^s
$
where $s=2$ if $p\geq3$ or  $s=3$ if $p=2$.
In that case, $G$ admits a ramification structure of size $(r_1, r_2)$ if and only if the following conditions hold:
\begin{enumerate}
\item
$r_1,r_2 \geq d+1$.
\item
If $p=3$ then $r_1, r_2 \geq 4$.
\item 
If $p=2$ then $r_1, r_2 \geq 5$.
\item
If $p=2$ and $|\{g^{2^{e-1}} \mid g \in G \}|=2^3$, then $(r_1,r_2)\neq (5,5)$, and furthermore if $e=1$, i.e. $G\cong C_2\times C_2 \times C_2$, then $r_1, r_2$ are not both odd.
\end{enumerate}
\end{thmA}

Note that the condition on the cardinality of the set $\{g^{p^{e-1}} \mid g \in G \}$ in Theorem A implies that if $G$ admits a ramification structure, then $d(G)\geq 2$ if $p\geq3$ or $d(G)\geq 3$ if $p=2$.

According to \cite[Theorem 3.18]{GP2}, if $G$ is an abelian $2$-group of exponent $2^e$ and $|G^{2^{e-1}}|=2^3$, then $G$ does not admit a ramification structure of size $(r_1,r_2)$ if $r_1$, $r_2$ are both odd. However, Theorem A shows that this statement is not true, and they can be both odd provided that $G\not\cong C_2\times C_2\times C_2$.

Theorem A applies to a wide family of $p$-groups, including regular $p$-groups (so, in particular, $p$-groups of exponent $p$ or of  nilpotency class less than $p$), powerful $p$-groups, and generalized $p$-central $p$-groups. A $p$-group is called generalized $p$-central if $p>2$ and $\Omega_1(G)\leq Z_{p-2}(G)$, or $p=2$ and $\Omega_2(G)\leq Z(G)$.

We want to remark that Theorem A is not valid for all finite $p$-groups. We will see that no condition on the cardinality of the set $\{g^{p^{e-1}} \mid g \in G \}$ can ensure the existence of ramification structures for the class of all finite $p$-groups.

On the other hand, if $G$ is a finite nilpotent group and $G_p$ is the Sylow $p$-subgroup of $G$, then we have $\bigcap_{p \mid |G| } S(G_p) \subseteq S(G)$, and  $S(G)\subseteq S(G_p)$ for odd primes $p$. However, it is not always true that $S(G)\subseteq S(G_2)$, even for abelian groups, contrary to what is implicit in the statement of Theorem 3.18 in \cite{GP2}. We give a counterexample to that in Example \ref{non-inherited size}. We fix this error in Theorem B.

\begin{thmB}
Let $G$ be a nilpotent group, and let $d=d(G)$. Let $G_p$ denote the Sylow $p$-subgroup of $G$ for every prime $p$ dividing $|G|$. Suppose that $\exp G_p=p^{e_p}$ and all $G_p$ are semi-$p^{e_p-1}$-abelian. Then $G$ admits a ramification structure if and only if all $G_p$ admit a ramification structure. In that case,  $(r_1, r_2) \in S(G)$ if and only if the following conditions hold:
\begin{enumerate}
\item 
$r_1, r_2 \geq d+1$.
\item
$(r_1, r_2) \in S(G_p)$ for $p$ odd.
\item
$(r_1, r_2) \in S(G_2)$ unless $G_2\cong C_2 \times C_2 \times C_2$.
\item
If $G_2\cong C_2 \times C_2 \times C_2$ then $r_1, r_2 \geq 5$ and $(r_1, r_2)\neq (5,5)$. Furthermore, if $G\cong C_2 \times C_2 \times C_2$ then $r_1, r_2$ are not both odd. 
\end{enumerate} 
\end{thmB}

\vspace{10pt}

\noindent
\textit{Notation.\/}
If $G$ is a finitely generated group, we write $d(G)$ for the minimum number of generators of $G$.
If $p$ is a prime and $G$ is a finite $p$-group, then $G^{p^i}=\langle g^{p^i} \mid g \in G \rangle$ and $\Omega_i(G)=\langle g \in G \mid g^{p^i}=1\rangle$. The exponent of $G$, denoted by $\exp G$, is the maximum of the orders of all elements of $G$.

\section{Finite $p$-groups}
Throughout this paper all groups will be finite.
In this section, we give the proof of Theorem A. Let us start with a general result related to lifting a spherical generating set of a factor group to the whole group.

\begin{pro}
\label{lifting process}
Let $G$ be a finite group and let $d=d(G)$. Suppose that  $N\unlhd G$ and $U=(\overline{x_1}, \dots, \overline{x_r})$ is a tuple of generators of $G/N$. Then the following hold:

\begin{enumerate}
\item 
If $r\geq d$ then there exist $n_1, \dots, n_r \in N$ such that  the tuple $T=( x_1n_1, \dots, x_rn_r)$ generates $G$.
\item
If $N\neq 1$, $r\geq d+1$ and $\overline{x_1}\dots\overline{x_r}=\overline{1}$, then we can choose $T$ to be a spherical system of generators of $G$.
\end{enumerate} 
\end{pro}

\begin{proof}
(i) See Proposition 2.5.4 in \cite{RZ}.

(ii) Assume first that $\overline{x_i}\neq \overline{1}$ for some $i=1, \dots, r$. For simplicity, we suppose that $\overline{x_r}\neq \overline{1}$. The equality $\overline{x_1}\dots\overline{x_r}=\overline{1}$ implies that $\langle \overline{x_1}, \dots, \overline{x_{r-1}} \rangle=G/N$.
Since $r-1 \geq d$ then by (i), there is a tuple $V=(z_1, \dots, z_{r-1})$ that generates $G$, where $z_i \in x_iN$ for $1\leq i \leq r-1$. Note that if  $\overline{x_j}=\overline{1}$, then it may happen that $z_j=1$. If this is the case, we take a nontrivial element in $N$ as $z_j$. Thus, $z_i \neq 1$ for $1\leq i \leq r-1$.

 If we call 
\[
T=\big(z_1, \dots, z_{r-1}, (z_1\dots z_{r-1})^{-1} \big),
\]
then clearly $T$ is a spherical system of generators of $G$. The only thing we have to show is that $(z_1\dots z_{r-1})^{-1} \in x_rN$. Observe that in $G/N$, we have $(\overline{z_1}\dots \overline{z_{r-1}})^{-1}=\overline{x_r}(\overline{z_1}\dots \overline{z_{r-1}} \ \overline{x_r})^{-1}=\overline{x_r}(\overline{x_1}\dots \overline{x_{r-1}} \ \overline{x_r})^{-1}=\overline{x_r}$. Thus, $(z_1\dots z_{r-1})^{-1} \in x_rN$. Since $\overline{x_r}\neq \overline{1}$,  this implies that $z_1\dots z_{r-1} \neq 1$.

Now suppose that $\overline{x_i}=\overline{1}$ for all $1\leq i \leq r$. Then $\overline{G}=\overline{1}$, and since $r\geq d+1$, we can take any spherical system of generators $T$ of $G$ of size $r$.
\end{proof}

Notice that in  part (ii) of Proposition \ref{lifting process}, we do not require that $U$ is a spherical system of generators of $G/N$. Therefore, as appears in the proof, some of $\overline{x_i} \in U$ might be  the identity of $G/N$.

We next state a theorem characterizing the possible sizes of ramification structures of elementary abelian $p$-groups. Before that we need the following lemma.

\begin{lem}
\label{basic operations}
Let $G$ be an elementary abelian $p$-group of rank $d$ with a ramification structure of size $(r_1, r_2)$. Then the following hold:
\begin{enumerate}
\item
$G$ admits a ramification structure of size $(r_1+1, r_2)$  if $p$ is odd, and of size $(r_1+2, r_2)$ if $p=2$.
\item
If $G^{*}$ is elementary abelian of rank $d+1$ and $r_1,r_2\geq d+2$, then $G^{*}$ admits a ramification structure of size $(r_1,r_2)$.
\end{enumerate}
\end{lem}

\begin{proof}
Let $(T_1, T_2)$ be a ramification structure  of size $(r_1, r_2)$ for $G$. We write $T_1=(x_1, x_2, \dots, x_{r_1})$.

We first prove (i). If
\begin{equation*}
T_1'
=
\begin{cases}
\big(x_1^2,x_2, \dots, x_{r_1}, x_1^{-1} \big)
&
\text{if \ $p$ is odd},
\\
\big(T_1, x_1, x_1\big)
&
\text{if \ $p=2$},
\end{cases}
\end{equation*}
then $(T_1', T_2)$ is a ramification structure as desired.

We next prove (ii). Let $G^{*}=G\times \langle y\rangle$ be an elementary abelian $p$-group of rank $d+1$.  Since $G$ is of rank $d$ and $r_1, r_2\geq d+2$, both $T_1$ and $T_2$ have at least two elements, say $a_1,b_1 \in T_1$ and $a_2, b_2 \in T_2$, that belong to the subgroup generated by the rest of the elements in $T_1$ and $T_2$, respectively. We modify $T_1, T_2$ to $T_1^{*}$ and $T_2^{*}$, by multiplying $a_1, a_2$ with $y$ and $b_1,b_2$ with $y^{-1}$. Then $(T_1^{*}, T_2^{*})$ is a ramification structure of size $(r_1,r_2)$ for $G^{*}$.
\end{proof}

Note that the roles of $r_1$ and $r_2$ are symmetric. Thus in Lemma \ref{basic operations}, $G$ also admits a ramification structure of size $(r_1, r_2+1)$ if $p$ is odd and of size $(r_1, r_2+2)$ if $p=2$.

\begin{thm}
\label{elementary abelian}
Let $G$ be an elementary abelian $p$-group of rank $d$ and let $r_1,r_2 \geq d+1$. Then $G$ admits a ramification structure of size $(r_1, r_2)$ if and only if the following conditions hold:
\begin{enumerate}
\item 
$d\geq 2$ if $p\geq 3$ or $d\geq 3$ if $p=2$.
\item
If $p=3$ then $r_1, r_2 \geq 4$.
\item
If $p=2$ then $r_1, r_2 \geq 5$, and furthermore if $d=3$ then $r_1, r_2$ are not both odd.
\end{enumerate}
\end{thm}

\begin{proof}
We first assume that $G$ admits a ramification structure  $(T_1, T_2)$ of size $(r_1, r_2)$. We already know that $d\geq 2$. If $p=2$ and $G\cong C_2 \times C_2$, then clearly $\Sigma(T_1) \cap \Sigma(T_2) \neq 1$, a contradiction. Thus, if $p=2$ then $d\geq 3$.

We next assume that $p=3$. We will show that $r_1, r_2 \geq 4$. Suppose, on the contrary, that $r_1=3$. Then $G\cong C_3 \times C_3$. If we write $T_1=(x_1,x_2, (x_1x_2)^{-1})$, then $\Sigma(T_1)$ contains $6$  different nontrivial elements of $G$. The other two nontrivial elements of $G$ are $x_1x_2^2$ and $x_1^2x_2^4$. Since they do not generate $G$, there is no ramification structure for $G$, which is a contradiction.

We now assume that $p=2$. We show that $r_1, r_2 \geq 5$. Suppose that $r_1=4$. Then $G\cong C_2 \times C_2 \times C_2$. We write $T_1= (x_1, x_2, x_3, (x_1x_2x_3)^{-1})$. Then $T_2$ can only contain $x_1x_2, x_1x_3$ and $x_2x_3$. However, $\langle x_1x_2, x_1x_3, x_2x_3\rangle \neq G$, again a contradiction.

Finally, we show that if  $G\cong C_2 \times C_2 \times C_2$ then $r_1, r_2$ are not both odd. Suppose that $r_1$ is odd. Then observe that $T_1$ contains at least $4$ different nontrivial elements. Otherwise, if $T_1$ has $3$ different nontrivial elements, say $u,v,t$, then $(u, v, t)$ is a minimal system of generators of $G$. Since the product of the elements of $T_1$ is equal to $1$, each of $u, v, t$ appears an even number of times in $T_1$, which is not possible since $r_1$ is odd. 

We now prove the converse. To this purpose, it is enough to find ramification structures of sizes $(3,3)$ or $(4,4)$ according as $p\geq 5$ or $p=3$ if $d=2$, of sizes $(5,6)$ or $(6,6)$ if $d=3$ and $p=2$, and finally of size $(5,5)$  if $d=4$ and $p=2$. Then by applying (i) and (ii) in Lemma \ref{basic operations} repeatedly, we get the result.

Let
$
G=\langle x_1\rangle \times \langle x_2\rangle\cong C_p\times C_p
$
where $p\geq 3$. If we take
\begin{equation*}
T_1
=
\begin{cases}
\big(x_1, x_2, (x_1x_2)^{-1}\big)
&
\text{if \ $p\geq 5$},
\\
\big(x_1, x_1^{-1}, x_2, x_2^{-1} \big)
&
\text{if \ $p=3$},
\end{cases}
\end{equation*}
and
\begin{equation*}
T_2
=
\begin{cases}
\big(x_1x_2^2, x_1x_2^4, (x_1^2x_2^6)^{-1}\big)
&
\text{if \ $p\geq 5$},
\\
\big(x_1x_2, (x_1x_2)^{-1}, x_1x_2^2, (x_1x_2^2)^{-1}\big)
&
\text{if \ $p=3$},
\end{cases}
\end{equation*}
then $(T_1,T_2)$ is a ramification structure for $G$ of size $(3,3)$ if $p\geq 5$, or of size $(4,4)$ if $p=3$. 

Now assume that $G=\langle x_1\rangle \times \langle x_2\rangle \times \langle x_3\rangle\cong C_2\times C_2\times C_2$. If we take 
\begin{equation*}
T_1
=
\begin{cases}
\big(x_1x_2, x_1x_3, x_2x_3, x_1x_2x_3, x_1x_2x_3\big)
&
\text{if \ $r_1=5$},
\\
\big(x_1x_2, x_1x_3, x_1x_2x_3, x_1x_2, x_1x_3, x_1x_2x_3 \big)
&
\text{if \ $r_1=6$},
\end{cases}
\end{equation*}
and $T_2=\big( x_1, x_2, x_3, x_1, x_2, x_3 \big)$, then $(T_1, T_2)$ is a ramification structure for $G$ of size $(5,6)$ or $(6,6)$. 

Finally if $p=2$ and
$G=\langle x_1\rangle \times \langle x_2\rangle \times \langle x_3\rangle\times \langle x_4\rangle$, then we take $T_1=\big(x_1,x_2, x_3, x_4, (x_1x_2x_3x_4)^{-1} \big)$ and 
$T_2=\big(x_1x_2, x_2x_3, x_3x_4, x_1x_2x_3, x_2x_3x_4  \big)$. Then clearly $(T_1, T_2)$ is a ramification for $G$ of size $(5,5)$. This completes the proof.
\end{proof}

Theorem \ref{elementary abelian} can  also be deduced from Theorem 3.18 in \cite{GP2} that characterizes abelian groups with ramification structures. However, note that the statement of that theorem corresponding to abelian $2$-groups is not true in general. According to Theorem 3.18 in \cite{GP2}, if $G$ is an abelian $2$-group of exponent $2^e$ with $|G^{2^{e-1}}|=2^3$ and $G$ admits a ramification structure of size $(r_1,r_2)$, then $r_1, r_2$ cannot be both odd. However, the next example shows that this is not necessarily the case. We fix this mistake  in Theorem \ref{case p=2}.

\begin{exa}
\label{counterexample}
Let $G=\langle a \rangle
\times \langle x \rangle 
\times \langle y \rangle
\times \langle z \rangle \cong C_2 \times C_4 \times C_4 \times C_4$. Now $\exp G=4$ and $|G^2|=2^3$. If we take
\[
T_1=(x, y, z, x^{-1}, y^{-1}, z^{-1}a, a),
\]
and
\[
T_2=(xya, xz, yz, xyz, xyza),
\]
then clearly $(T_1, T_2)$ is a ramification structure for $G$ of size $(7, 5)$.
\end{exa}

We next see that the existence of ramification structures for a group of exponent $p$ can be deduced from Theorem \ref{exponent p}.

\begin{thm}
\label{exponent p}
Let $G$ be a $p$-group of exponent $p$.  Then $G$ admits a ramification structure of size $(r_1,r_2)$ if and only if $G/\Phi(G)$ admits a ramification structure of size $(r_1,r_2)$.
\end{thm}

\begin{proof}
Note that if $p=2$ then $G$ is an elementary abelian $2$-group, and hence $G$ coincides with $G/\Phi(G)$.  Thus we assume that $p\geq3$.  We first show that
if $G/\Phi(G)$ admits a ramification structure $(U_1, U_2)$ of size $(r_1,r_2)$, then so does $G$. 

Consider a lift of $(U_1, U_2)$ to $G$, say $(T_1, T_2)$, such that $T_1$ and $T_2$ are spherical systems of generators of $G$. Since $\exp G= p$, all elements in $T_1$ and $T_2$ are of order 
$p$. We claim that $(T_1, T_2)$ is a ramification structure of size $(r_1,r_2)$ for $G$. Suppose, on the contrary, that there are $a\in T_1$ and $b\in T_2$ such that
$\langle a\rangle^g=\langle b\rangle$ for some $g\in G$. Since $G/\Phi(G)$ is abelian, we get $\langle \overline{a}\rangle=\langle  \overline{b}\rangle$, which is a contradiction.
	
Let us now prove the converse. Assume that $G$ admits a ramification structure of size $(r_1,r_2)$. Note that $G/\Phi(G)$ has rank at least $2$. Then by Theorem \ref{elementary abelian}, any elementary abelian $p$-group of rank $\geq 2$ for $p\geq 5$ admits a ramification structure of size $(r_1,r_2)$ if $r_1, r_2 \geq 3$.
	
Finally we assume that $p=3$. According to Theorem \ref{elementary abelian}, we only need to prove that $G$ does not admit a ramification structure with $r_1=3$. By way of contradiction, it follows that $G$ is a $2$-generator group with $\exp G=3$. Then  \cite[14.2.3]{rob} implies that $G$ is of order $3^3$. Observe that each element in $T_1$ falls into a different maximal subgroup of $G$. Since $G$ has $4$ maximal subgroups and not all elements in $T_2$ fall into the same maximal subgroup, it then follows that there are elements in $T_1$ and $T_2$, say $a\in T_1$ and $b\in T_2$, which are in the same maximal subgroup.  Then we have
\[
b=a^ic,
\]
for some $c\in \Phi(G)=G'$ and for $i \in \{1,2\}$. Since $|G|=3^3$ and $a^i$ is a generator of $G$, we can write $c=[a^i, g]$ for some $g\in G$. It then follows that $b=(a^i)^g$, a contradiction.
\end{proof}

\vspace{10pt}

We now introduce a property which is essential to our result, and then we describe some families of finite $p$-groups satisfying this property.

Let $G$ be a finite $p$-group, and let $i\ge 1$ be an integer.
Following Xu \cite{xu}, we say that $G$ is \emph{semi-$p^i$-abelian\/} if the following condition holds:
\begin{equation}
\label{semi-pi-abelian}
x^{p^i} = y^{p^i}
\quad
\text{if and only if}
\quad
(xy^{-1})^{p^i} = 1.
\end{equation}
If $G$ is semi-$p^i$-abelian, then we have \cite[Lemma 1]{xu}:
\begin{enumerate}
\item[(SA1)]
$\Omega_i(G)=\{x\in G\mid x^{p^i}=1\}$.
\item[(SA2)]
$|G:\Omega_i(G)|=|\{x^{p^i}\mid x\in G\}|$.
\end{enumerate}
If $G$ is semi-$p^i$-abelian for every $i\ge 1$, then $G$ is called \emph{strongly semi-$p$-abelian\/}.

By \cite[Theorem 3.14]{suz2}, regular $p$-groups are strongly semi-$p$-abelian.
On the other hand, by Lemma 3 in \cite{fer}, a powerful $p$-group of exponent $p^e$ is semi-$p^{e-1}$-abelian.
%Also for odd $p$, the groups in which $\Omega_1(G)\le Z(G)$, i.e.\ $p$-central $p$-groups, are strongly semi-$p$-abelian, as follows from Theorem 1 in \cite{xu}.
Furthermore, by Theorem 2.2 in \cite{FG}, generalized $p$-central $p$-groups, i.e.\ groups in which $\Omega_1(G)\leq Z_{p-2}(G)$ for odd $p$, or $\Omega_2(G)\leq Z(G)$ for $p=2$, are strongly semi-$p$-abelian.

\vspace{10pt}

Before we proceed to prove Theorem A, we need the following lemma.

\begin{lem}
\label{induced rs}
Let $G$ be a $p$-group of exponent $p^e$ and let $d=d(G)$. Suppose that $G$ is semi-$p^{e-1}$-abelian. Then the following hold:
\begin{enumerate}
\item 
If $(T_1, T_2)$ is a ramification structure for $G$, then 
$\big(\overline{T}_1\smallsetminus \{\overline{1}\}, \overline{T}_2\smallsetminus \{\overline{1}\}\big)$ is a ramification structure for $G/\Omega_{e-1}(G)$.
\item
If $(U_1,U_2)$ is a ramification structure of size $(r_1, r_2)$ for $G/\Omega_{e-1}(G)$ and $r_1, r_2\geq d+1$, then there is a lift of $(U_1,U_2)$ to $G$ which is a ramification structure of size $(r_1, r_2)$ for $G$.
\end{enumerate}
\end{lem}

\begin{proof}
We first prove (i) by way of contradiction. Note that $G/\Omega_{e-1}(G)$  is of exponent $p$. 
Suppose that there are 
$\overline{a} \in \overline{T}_1 \smallsetminus \{\overline{1}\}$ and 
$\overline{b} \in \overline{T}_2\smallsetminus \{\overline{1}\}$ such that 
$ \langle \overline{a} \rangle= \langle \overline{b}\rangle^{\overline{g}}$ for some  $\overline{g} \in G/\Omega_{e-1}(G)$, i.e. $\overline{b}^{\overline{g}}=\overline{a}^i$ for some $i$ not divisible by $p$. Then we have $b^ga^{-i} \in \Omega_{e-1}(G)$, and consequently $(b^ga^{-i})^{p^{e-1}}=1$, by (SA1).  Since $G$ is semi-$p^{e-1}$-abelian, we get $(b^g)^{p^{e-1}}=a^{ip^{e-1}}$. This is a contradiction, since both $a$ and $b$ are of order $p^e$ and $\langle a\rangle \cap \langle b\rangle^{g}=1$.

We next prove (ii). By  part (ii) of Proposition \ref{lifting process}, we can take a lift of $(U_1,U_2)$ to $G$, say $(T_1, T_2)$, such that $T_1$ and $T_2$ are spherical systems of generators of $G$. Observe that all elements in $T_1$ and $T_2$ are of order $p^e$. We next show that $T_1$ and $T_2$ are disjoint. Suppose, on the contrary, that there are $a\in T_1$ and $b\in T_2$ such that
\[
\langle a^{p^{e-1}}\rangle^g=\langle b^{p^{e-1}}\rangle,
\]
for some $g\in G$, i.e $(a^g)^{p^{e-1}}=b^{ip^{e-1}}$ for some integer $i$ not divisible by $p$. Since $G$ is semi-$p^{e-1}$-abelian, then  $a^gb^{-i} \in \Omega_{e-1}(G)$, and consequently, $\langle \overline{a}\rangle^{\overline{g}}=\langle \overline{b}\rangle$ in $G/\Omega_{e-1}(G)$, which is a contradiction since $(U_1,U_2)$ is a ramification structure for $G/\Omega_{e-1}(G)$.
\end{proof}

We are now ready to prove Theorem A. We deal separately with the cases $p\geq 3$ and $p=2$.

\begin{thm}
\label{case p>2}
Let $G$ be a $p$-group of exponent $p^e$ with $p\geq 3$, and let $d=d(G)$. Suppose that $G$ is semi-$p^{e-1}$-abelian. Then $G$ admits a ramification structure if and only if 
$|\{g^{p^{e-1}} \mid g \in G \}| \geq p^2$.
In that case, $G$ admits a ramification structure of size $(r_1,r_2)$ if and only if $r_1, r_2 \geq d+1$, and also $r_1, r_2\geq 4$ provided that $p=3$.
\end{thm}

\begin{proof}
	
We first assume that $G$ admits a ramification structure  $(T_1, T_2)$. By (SA2), the cardinality of the set $X=\{g^{p^{e-1}} \mid g \in G \}$ is a power of $p$. Suppose that $|X|=p$. It then follows that the subgroup $G^{p^{e-1}}$ is cyclic of order $p$.
Note that by (SA1), we have $\exp \Omega_{e-1}(G)=p^{e-1}$.
Then there are elements $a\in T_1$ and $b\in T_2$ such that $o(a)=o(b)=p^e$. Thus,
\[
G^{p^{e-1}}= \langle a^{p^{e-1}}\rangle=\langle b^{p^{e-1}}\rangle,
\]
which is a contradiction.

We next prove that if $p=3$ and $G$ admits a ramification structure of size $(r_1, r_2)$, then
$r_1,r_2 \geq 4$. Suppose, by way of contradiction, that $r_1=3$. Then since  $|X| \geq 3^2$, we have $|G/\Omega_{e-1}(G)|\geq 3^2$, by (SA2). Part (i) of Lemma \ref{induced rs} implies that $G/\Omega_{e-1}(G)$ admits a ramification structure of size $(r,s)$ where $r\leq r_1\leq3$. However, according to Theorems \ref{elementary abelian} and \ref{exponent p} this is not possible. 

Now assume that $|X| \geq p^2$. Let us use the bar notation $\overline{G}$ for the factor group $G/\Omega_{e-1}(G)$. Then  $|\overline{G}|\geq p^2$ and $d(\overline{G})\geq 2$.
It follows from Theorems \ref{elementary abelian} and \ref{exponent p} that $\overline{G}$ admits a ramification structure  of size $(r, s)$ for all $r,s \geq d(\overline{G})+1$, and $r,s \geq 4$ provided that $p=3$. If we take $r_1, r_2 \geq d+1\geq d(\overline{G})+1$, and $r_1, r_2 \geq4$ provided that $p=3$, then part (ii) of Lemma \ref{induced rs} implies that $G$ admits a ramification structure of size $(r_1, r_2)$.  This completes the proof.
\end{proof}

We next deal with the prime $2$.

\begin{thm}
\label{case p=2}
Let $G$ be a $2$-group of exponent $2^e$, and let $d=d(G)$. Suppose that $G$ is semi-$2^{e-1}$-abelian. Then  $G$ admits a ramification structure if and only if 
$|\{g^{2^{e-1}} \mid g \in G \}| \geq 2^3$.
In that case, $G$ admits a ramification structure of size $(r_1, r_2)$ if and only if the following conditions hold:
\begin{enumerate}
\item
$r_1, r_2 \geq d+1$.
\item 
$r_1, r_2 \geq 5$.
\item
If $|\{g^{2^{e-1}} \mid g \in G \}|=2^3$, then $(r_1,r_2)\neq (5,5)$, and furthermore if $e=1$, i.e. $G\cong C_2\times C_2 \times C_2$, then $r_1, r_2$ are not both odd.
\end{enumerate}
\end{thm}

\begin{proof}
We first assume that $G$ admits a ramification structure.
Suppose that $X=\{g^{2^{e-1}} \mid g \in G \}$ is of cardinality at most $2^2$, so that $|G:\Omega_{e-1}(G)|\leq 2^2$. Then according to Theorem \ref{elementary abelian},  $G/\Omega_{e-1}(G)$ does not admit a ramification  structure. Thus, $G$ has no ramification structure, as follows from Lemma \ref{induced rs}(i). This is a contradiction. So we have $|X| \geq2^3$.

If the ramification structure for $G$ is of size $(r_1, r_2)$, then we have $r_1, r_2 \geq d+1$. By Theorem \ref{elementary abelian}, ramification structures of $G/\Omega_{e-1}(G)$ have size $(r,s)$ where $r,s\geq 5$, and furthermore $r,s$ are not both odd if $|G/\Omega_{e-1}(G)|=2^3$. Hence, by part (i) of Lemma \ref{induced rs},  we have $r_1, r_2 \geq 5$ and  furthermore, if  $|G/\Omega_{e-1}(G)|=2^3$ then $(r_1, r_2) \neq (5,5)$. Finally if 
$G\cong C_2\times C_2 \times C_2$ then $r_1, r_2$ are not both odd, by Theorem \ref{elementary abelian}.

We now work under the assumption $|X| \geq2^3$. Suppose that $r_1, r_2 \geq d+1$, $r_1, r_2 \geq 5$ and furthermore that $r_1, r_2$ are not both odd if $|X|=2^3$.  Then by Theorem \ref{elementary abelian}, $G/\Omega_{e-1}(G)$ admits a ramification structure of size $(r_1, r_2)$. Lemma \ref{induced rs}(ii) implies that $G$ admits a ramification structure of size $(r_1, r_2)$.

It remains to prove that if $r_1, r_2\geq5$, $(r_1,r_2)\neq (5,5)$ and both $r_1, r_2$ are odd, then $G$ admits a ramification structure of size $(r_1, r_2)$ under the assumptions $|X|=2^3$ and $e\geq2$. We may assume that $r_2\geq 7$. Then  $G/\Omega_{e-1}(G)$ admits a ramification structure of size $(r_1, r_2-1)$.

Since $G/G^2$ is elementary abelian of rank $d$ and $G/\Omega_{e-1}(G)$ is of rank $3$, we have
$ \Omega_{e-1}(G)/G^2$ is of rank $d-3$. We take a generating set $\{n_1, \dots, n_{d-3}\}$ of $\Omega_{e-1}(G)$ modulo $G^2$. Call $n=n_1\dots n_{d-3}$ and let $o(n)=2^k<2^e$. If $1\neq n^{2^{k-1}}=x^{2^{e-1}}$ for some $x\in G$, then since $x\notin \Omega_{e-1}(G)$ we take a generating set of $G/\Omega_{e-1}(G)$ containing $\overline{x}$, say 
$G/\Omega_{e-1}(G)=\langle \overline{x} \rangle \times \langle \overline{y} \rangle \times \langle \overline{z} \rangle$. Otherwise, if $n^{2^{k-1}} \neq g^{2^{e-1}}$ for any $g \in G$, then we take any generating set of $G/\Omega_{e-1}(G)$.

Now consider the following ramification structure of $G/\Omega_{e-1}(G)$:
\[
U_1=\big( \overline{xy},  \overline{yz}, \overline{xz}, \overline{xyz}, \overline{xyz}, \overline{xy}, \dots, \overline{xy} \big) \ \ \
\text{and}
\]
\[
U_2=\big( \overline{x}, \overline{y}, \overline{z}, \overline{x}, \overline{y}, \overline{z}, \overline{x}, \dots, \overline{x}\big),
\]
where $|U_1|=r_1$ and $|U_2|=r_2-1$.
Since $r_1 \geq d+1$,  by part (ii) of Proposition \ref{lifting process}, we take a lift $T_1$ of $U_1$ so that $T_1$ is a spherical system of generators of $G$. Then consider the following lift of $U_2$ to $G$:
\[
T_2=\big(x, y, z, xn_1, yn_2, zn_3, xn_4, \dots, xn_{d-3}, x, \dots, x  \big),
\]
where $|T_2|=r_2-1$. Clearly, $T_2$ generates $G$. Observe that the product of all components of $T_2$ is $n$ modulo $G^2$, i.e. the product is equal to  $wn$ for some $w\in G^2$. Now consider the following tuple:
\[
T_2^*=\big(w^{-1}x, y, z, xn_1, yn_2, zn_3, xn_4, \dots, xn_{d-3}, x, \dots, x, n^{-1} \big),
\]
where $|T_2|=r_2$. Since $w\in G^2=\Phi(G)$, it follows that $T_2^*$ generates $G$ and furthermore, it is spherical. Our claim is that $(T_1, T_2^*)$ is a ramification structure of size $(r_1, r_2)$ for $G$.

Notice that all elements in $T_1 \cup T_2^*$ are of order $2^e$ except $n^{-1}$. Then by using the same argument in the proof of part (ii) of Lemma \ref{induced rs}, we conclude that 
$\langle a \rangle^g \cap \langle b\rangle=1$ for any $g\in G$,  $a\in T_1$ and $b\in T_2^ *\smallsetminus \{n^{-1}\}$. On the other hand, if $n^{2^{k-1}}=x^{2^{e-1}}$ then  since $\langle x^{2^{e-1}}\rangle \neq \langle a^{2^{e-1}}\rangle^g$ for any $g\in G$ and $a\in T_1$, we have
$\langle n\rangle \cap \Sigma(T_1)=1$. Otherwise, if $n^{2^{k-1}} \neq g^{2^{e-1}}$ for any $g\in G$, then clearly $\langle n\rangle \cap \Sigma(T_1)=1$. This completes the proof.
\end{proof}

We close this section by showing that the assumption of being semi-$p^{e-1}$-abelian is essential in Theorem A. As we next see, for a general finite $p$-group $G$, the cardinality of the set $\{g^{p^{e-1}} \mid g\in G \}$ does not control the existence of ramification structures for $G$. To this purpose, we will work with $2$-generator $p$-groups constructed in \cite{FG}. For more details, we suggest readers to see pages 11-13 of \cite{FG}.

\begin{lem}
\label{beauville to rs}
Let $G$ be a Beauville group. Then $G$ admits a ramification structure of size $(r_1, r_2)$ for any $r_1, r_2\geq 3$.
\end{lem}

\begin{proof}
Assume that $G$ is a Beauville group, that is, it admits a ramification structure $(U_1, U_2)$ of size $(3,3)$. Let $U_1=\big(x_1, y_1, (x_1y_1)^{-1}\big)$ and let  $U_2=\big(x_2, y_2, (x_2y_2)^{-1}\big)$. Consider the following tuples:
\[
T_1=\big( x_1, y_1, y_1^{-1}, x_1^{-1}\big) \ \ \
\text{or} \ \ \
T_1=U_1,
\]
and
\[
T_2=\big( x_2, y_2, y_2^{-1}, x_2^{-1}\big) \ \ \
\text{or} \ \ \
T_2=U_2.
\]
By adding $x_1, x_1^{-1}$ to $T_1$ and
$x_2, x_2^{-1}$ to $T_2$ repeatedly, we obtain a pair of spherical systems of generators 
$(T_1^*, T_2^*)$ for $G$ of size $(r_1, r_2)$ for any $r_1, r_2 \geq 3$. Then since 
$(U_1, U_2)$ is a ramification structure for $G$, so does $(T_1^*, T_2^*)$.
\end{proof}

The following result shows that the `only if' part of Theorem A fails for a general finite $p$-group.

\begin{pro}
Let $p\ge 5$ be a prime. Then there exists a $p$-group $G$ such that:
\begin{enumerate}
\item
If $\exp G=p^e$ then $|\{g^{p^{e-1}} \mid g\in G\}|=p$.
\item
$G$ admits a ramification structure of size $(r_1, r_2)$ for any $r_1, r_2\geq 3$.
\end{enumerate}
\end{pro}

\begin{proof}
In the proof of Corollary 2.12 in \cite{FG}, it was shown that there exists a Beauville $p$-group $G$ with $\exp G=p^e$ such that $|G^{p^{e-1}}|=p$. It  then follows that $|\{g^{p^{e-1}} \mid g\in G\}|=p$ and hence (i) holds. Since $G$ is a Beauville group, (ii) readily follows from Lemma \ref{beauville to rs}. 
\end{proof}

Finally, the following result shows that for every power of $p$, there is a $p$-group $G$ such that the cardinality of the set $\{g^{p^{e-1}} \mid g\in G\}$ is exactly that power and $G$ does not admit a ramification structure.

\begin{pro}
For every prime $p\ge 5$, and positive integer $m$, there exists a $p$-group $G$ such that:
\begin{enumerate}
\item
If $\exp G=p^e$ then $|\{g^{p^{e-1}} \mid g\in G\}|=p^m$.
\item
$G$ does not admit a ramification structure.
\end{enumerate}
\end{pro}

\begin{proof}
Consider the group $G$ in the second part of the proof of Corollary 2.12 in \cite{FG}. Then $G$ is a $2$-generator $p$-group $G$ with $\exp G=p^e$ such that  $|G^{p^{e-1}}|=p^m$ for some $m$. One can also observe from the proof that the subgroup $G^{p^{e-1}}$ coincides with the set $\{g^{p^{e-1}} \mid g\in G\}$. Furthermore, it was shown that for every pair of generating sets $(x_1, y_1)$ and $(x_2, y_2)$, there are elements, say $x_1$ and $x_2$, such that $\langle x_1^i\rangle=\langle x_2^j\rangle\neq 1$ for some integers $i ,j$.
Thus, $G$ does not admit a ramification structure.
Furthermore, Corollary 2.13 in \cite{FG} implies that $m$ can be any positive integer.
\end{proof}

\section{Finite nilpotent groups }
In this section, we prove Theorem B. We give the possible sizes of ramification structures for nilpotent groups whose Sylow $p$-subgroups are semi-$p^{e-1}$-abelian if the exponent is $p^e$. To this purpose, we need the following result regarding a direct product of groups of coprime order.

\begin{pro}
\label{direct product}
Let $G$ and $G^{*}$ be groups of coprime order. Then the following hold:
\begin{enumerate}
\item 
If $G$ and $G^{*}$ admit ramification structures of size $(r_1, r_2)$ and $(r_1^*,r_2^*)$, respectively, then $G\times G^{*}$ admits a ramification structure of size  $(r, s)$ where $r=\max\{r_1, r_1^*\}$ and $s= \max\{r_2, r_2^*\}$.
\item
If $G\times G^*$ admits a ramification structure of size $(r, s)$, then $G$ and $G^*$ admit ramification structures of size $(r_1, r_2)$ and $(r_1^*,r_2^*)$, respectively, for some $r_1, r_1^* \leq r$ and $r_2, r_2^* \leq s$. Furthermore, if $G$ is of odd order, we also have $r_1=r$ and $r_2=s$.
\end{enumerate}
\end{pro}

\begin{proof}
We first prove (i).
Assume that $(T_1, T_2)$ and $ (T_1^{*}, T_2^{*})$ are ramification structures of size $(r_1, r_2)$ and $(r_1^*, r_2^*)$ for $G$ and $G^{*}$, respectively. Let $r=\max\{r_1, r_1^*\}$ and $s= \max\{r_2, r_2^*\}$. Then by adding as many times the identity as needed to $T_1, T_2, T_1^*$ and $T_2^*$, we obtain $U_1, U_2, U_1^*$ and $U_2^*$ where $|U_1|=|U_1^*|=r$ and $|U_2|=|U_2^*|=s$. Let
\[
U_1=(x_1, \dots, x_r ) \ \ \
\text{and} \ \ \
U_2=(y_1, \dots, y_s),
\]
\[
U_1^*=(x_1^*, \dots, x_r^* ) \ \ \
\text{and} \ \ \
U_2^*=(y_1^*, \dots, y_s^*).
\]
Then let 
\[
A_1=\big( (x_1, x_1^{*}), \dots, (x_r, x_r^{*})\big) \ \ \
\text{and} 
\]
\[
A_2=\big( (y_1,y_1^{*}), \dots, (y_s, y_s^{*})\big).
\]
Observe that since $G$ and $G^{*}$ have coprime order, both $A_1$ and $A_2$ generate $G\times G^{*}$. We will see that $(A_1, A_2)$ is a ramification structure for $G\times G^{*}$. Otherwise, there exist $(a,a^{*}) \in A_1$ and $(b,b^{*}) \in A_2$ such that
\[
\langle (a,a^{*}) \rangle^{(g,g^{*})} \cap \langle (b,b^{*}) \rangle
\neq
\{(1,1)\}, 
\]
for some $(g,g^{*}) \in G \times G^{*}$. It then follows that either $\langle a \rangle^g \cap \langle b\rangle \neq 1$ or $\langle a^{*}\rangle^{g^{*}} \cap \langle b^{*}\rangle \neq 1$, which is a contradiction. 

Let us now prove (ii). Assume that 
\[
A_1=\big( (x_1,x_1^{*}), \dots, (x_{r}, x_{r}^{*}) \big) \ \ \
\text{and} \ \ \
A_2=\big( (y_1,y_1^{*}), \dots, (y_{s}, y_{s}^{*}) \big)
\]
form a ramification structure of size $(r, s)$ for $G\times G^{*}$. Assume that after deleting the identity element in $(x_1, \dots, x_{r})$ and $(y_1, \dots, y_{s})$ we get $T_1=(z_1, \dots, z_{r_1})$ and $T_2=(t_1, \dots, t_{r_2})$ for some $r_1\leq r$ and $r_2\leq s$.  We claim  that $(T_1, T_2)$ is a ramification structure of size $(r_1, r_2)$ for $G$. The same arguments apply to $G^{*}$.
For every $(a,a^{*})\in A_1$ and $(b,b^{*}) \in A_2$ we have
\begin{equation}
\label{intersection in product}
\langle (a,a^{*})\rangle^{(g,g^{*})} \cap \langle (b,b^{*})\rangle=\{(1,1)\},
\end{equation}
for all $(g, g^{*})\in G\times G^{*}$. Let $|G|=l$ and $|G^{*}|=m$, where $\gcd(l,m)=1$. Then by equation (\ref{intersection in product}), we get
\[
\langle ((a^m)^g, 1)\rangle \cap \langle (b^m, 1)\rangle=\{(1,1)\},
\]
and  hence $\langle a^m \rangle^g \cap \langle b^m \rangle=1 $. Since $\gcd(l,m)=1$, it then follows that 
$\langle a \rangle^g \cap \langle b \rangle=1$. 

Finally we assume that $G$ is of odd order. If $r-r_1$ is even, then we take 
$T_1=(z_1, \dots, z_{r_1}, z_1, z_1^{-1}, \dots, z_1, z_1^{-1})$. Now suppose that $r-r_1$ is odd. Since $G$ is of odd order, we have $o(z_1)\neq 2$. Then in this case we take 
$T_1=(z_1^2, z_1^{-1}, z_2, \dots, z_{r_1}, z_1, z_1^{-1}, \dots, z_1, z_1^{-1})$.  In both cases, $T_1$ is a spherical system of generators of $G$ of size $r$. By using the same arguments, we can make $|T_2|=s$. Then by the previous paragraph, $(T_1,T_2)$ is a ramification structure of size $(r, s)$ for $G$, as desired. This completes the proof.
\end{proof}

The following proposition is easily deduced from Proposition \ref{direct product}.

\begin{pro}
\label{nilpotent general}
Let $G$ be a nilpotent group. Then
\begin{enumerate}
\item 
$G$ admits a ramification structure if and only if all Sylow $p$-subgroups of $G$ admit a ramification structure.
\item 
If $G$ is of odd order, then $G$ admits a ramification structure of size $(r_1, r_2)$ if and only if all Sylow $p$-subgroups of $G$ admit a ramification structure of size $(r_1, r_2)$.
\end{enumerate}
\end{pro}

In order to characterize abelian groups with ramification structures, Garion and Penegini \cite{GP2} reduced the study to their Sylow $p$-subgroups. However, as far as the sizes of ramification structures are concerned, this reducing argument is not correct in general. More precisely, if $G$ is an abelian group of even order, then the size of a ramification structure of $G$ need not be inherited by the Sylow $2$-subgroup of $G$, as we see in the next example. We fix this mistake in Theorem \ref{nilpotent special}.

\begin{exa}
\label{non-inherited size}
Let $G=\langle a \rangle
\times \langle b \rangle 
\times \langle c \rangle
 \cong C_6 \times C_6 \times C_2$. If we take
\[
T_1=(a, b, c, b^{-1}, (ac)^{-1}),
\]
and
\[
T_2=(ab, ab, (ab)^{-2}, abc, (abc)^{-1}, a^2bc, (a^2bc)^{-1}),
\]
then $(T_1, T_2)$ is a ramification structure of size $(5, 7)$ for $G$ . However, the Sylow $2$-subgroup of $G$, which is $C_2\times C_2 \times C_2$, does not admit a ramification structure of size $(5,7)$.
\end{exa}

We close the paper by proving Theorem B.

\begin{thm}
\label{nilpotent special}
Let $G$ be a nilpotent group, and let $d=d(G)$. Let $G_p$ denote the Sylow $p$-subgroup of $G$ for every prime $p$ dividing $|G|$. Suppose that $\exp G_p=p^{e_p}$ and all $G_p$ are semi-$p^{e_p-1}$-abelian. Then $G$ admits a ramification structure of size $(r_1, r_2)$ if and only if the following conditions hold:
\begin{enumerate}
\item 
$r_1, r_2 \geq d+1$.
\item
$(r_1, r_2) \in S(G_p)$ for $p$ odd.
\item
$(r_1, r_2) \in S(G_2)$ unless $G_2\cong C_2 \times C_2 \times C_2$.
\item
If $G_2\cong C_2 \times C_2 \times C_2$ then $r_1, r_2 \geq 5$ and $(r_1, r_2)\neq (5,5)$. Furthermore, if $G\cong C_2 \times C_2 \times C_2$ then $r_1, r_2$ are not both odd. 
\end{enumerate} 
\end{thm}

\begin{proof}
We first assume that $(r_1, r_2)\in S(G)$. We know that (i) holds, and by Proposition \ref{direct product}(ii), we have (ii).
We next assume that $G_2 \neq 1$. Then again by Proposition \ref{direct product}(ii), $G_2$ admits a ramification structure of size $(r,s)$ for some $r\leq r_1$ and $s\leq r_2$. Then by Theorem \ref{case p=2},  $r,s \geq 5$, and furthermore $(r,s)\neq (5,5)$ if 
$|\{g^{e_2-1} \mid g\in G_2\}|=2^3$. This implies that $r_1, r_2 \geq 5$, and furthermore 
$(r_1, r_2)\neq (5,5)$ if $|\{g^{e_2-1} \mid g\in G_2\}|=2^3$.  Then the first part of (iv) follows, and (iii) follows from Theorem \ref{case p=2}. Finally if 
$G\cong C_2\times C_2 \times C_2$ then $r_1, r_2$ are not both odd, by Theorem \ref{elementary abelian}.

Conversely, assume that conditions (i)-(iv) hold. Then all $G_p$ admit a ramification structure of size $(r_1, r_2)$ unless $G_2\cong C_2 \times C_2 \times C_2$. Thus, if $G_2 \not \cong  C_2 \times C_2 \times C_2$, by Proposition \ref{direct product}(i), we conclude that $G$ admits a ramification structure of size $(r_1, r_2)$.

Finally we assume that conditions (i)-(iv) hold and $G_2=\langle x \rangle \times \langle y \rangle \times \langle z \rangle \cong C_2 \times C_2 \times C_2$. If $G=G_2$ then we already know the result, by Theorem \ref{case p=2}. Thus, we assume that $G\neq G_2$. Let $R$ be the direct product of the Sylow $p$-subgroups of $G$ for all odd primes $p$ dividing $|G|$. Then Proposition \ref{nilpotent general}(ii), together with condition (ii), implies that $R$ admits a ramification structure of size $(r_1, r_2)$.

If $r_1, r_2$ are not both odd, then $G_2$ also admits a ramification structure of size $(r_1, r_2)$. Otherwise, if both $r_1, r_2$ are odd, then we may assume that $r_2\geq 7$, and thus $G_2$ admits a ramification structure of size $(r_1, r_2-1)$, by Theorem \ref{elementary abelian}. Then in both cases, Proposition \ref{direct product}(i) implies that $G=R\times G_2$ admits a ramification structure of size $(r_1, r_2)$. This completes the proof.
\end{proof}

\section*{Acknowledgments}
I would like to thank G. A. Fern\'andez-Alcober for intense discussions and his feedbacks.

\end{document}